\documentclass[10pt,reqno]{amsart}
\usepackage{amsmath}
\usepackage{amsthm}
\usepackage{amssymb}
\usepackage{amsfonts}
\usepackage{latexsym}
\usepackage{hyperref}

\renewcommand{\phi}{\varphi}

\newcommand{\micro}[1]{\fontsize{#1}{#1}\selectfont}

\newcommand{\C}{\mathbb{C}}
\newcommand{\h}{\mathcal{H}}

\newcommand{\Z}{\mathbb{Z}}

\newcommand{\R}{\mathbb{R}}

\newcommand{\inner}[1]{\langle #1 \rangle}

\newcommand{\minimatrix}[4]{\begin{pmatrix} #1 & #2 \\ #3 & #4 \end{pmatrix}  }
\newcommand{\megamatrix}[9]{\begin{pmatrix} #1 & #2 & #3 \\ #4 & #5 & #6 \\ #7 & #8 & #9\end{pmatrix}  }

\newcommand{\twovector}[2]{\begin{pmatrix} #1\\#2 \end{pmatrix} }
\newcommand{\threevector}[3]{\begin{pmatrix} #1\\#2\\#3 \end{pmatrix} }

\newtheorem{Theorem}{Theorem}

\newtheorem{Lemma}{Lemma}
\theoremstyle{definition}

\newtheorem{Example}{Example}

\allowdisplaybreaks
\begin{document}
\title[Unitary equivalence to a complex symmetric matrix]{Unitary equivalence to a complex symmetric matrix:  a modulus criterion}
\author[S.R.~Garcia]{Stephan Ramon Garcia}
	\email{Stephan.Garcia@pomona.edu}
	\urladdr{http://pages.pomona.edu/\textasciitilde sg064747}

\author[D.E.~Poore]{Daniel E.~Poore}
	\email{dep02007@mymail.pomona.edu}

\author[M.K.~Wyse]{Madeline K.~Wyse}
	\email{mkw02007@mymail.pomona.edu}

\address{Department of Mathematics\\
	Pomona College\\
	610 North College Avenue\\
	Claremont, California\\
	91711}

    
	\thanks{The first author was partially supported by National Science Foundation Grant DMS-0638789.}

\begin{abstract}
	We develop a procedure for determining whether a square complex matrix
	is unitarily equivalent to a complex symmetric (i.e., self-transpose) matrix.  Our approach has
	several advantages over existing methods \cite{Balayan,Tener}.
	We discuss these differences and present a number of examples.
\end{abstract}

\maketitle

\section{Introduction}

	Following \cite{Tener}, we say that a matrix $T \in M_n(\C)$ is \emph{UECSM} if it is 
	unitarily equivalent to a complex symmetric (i.e., self-transpose) matrix.  Our primary motivation for studying this concept
	stems from the emerging theory of complex symmetric operators
	on Hilbert space \cite{Chalendar, Chevrot, CRW, CCO,CSOA,CSO2, SNCSO,Gilbreath, Sarason,Wang,Z}.
	A bounded operator $T$ on a separable complex Hilbert space $\h$
	is called a \emph{complex symmetric operator} if $T = CT^*C$ for some
	conjugation $C$ (a conjugate-linear, isometric involution) on $\h$.  The terminology stems from the fact that
	the preceding condition is equivalent to insisting that $T$ has
	a complex symmetric matrix representation with respect to some
	orthonormal basis \cite[Sect.~2.4-2.5]{CCO}.  
	Thus the problem of determining whether a given matrix is UECSM is equivalent to
	determining whether that matrix represents a complex symmetric operator with respect to some
	orthonormal basis.  Equivalently, $T$ is UECSM if and only if $T$ belongs to the unitary orbit 
	of the complex symmetric matrices in $M_n(\C)$.

	Since every $n \times n$ complex matrix is \emph{similar} to
	a complex symmetric matrix \cite[Thm.~4.4.9]{HJ} (see also \cite[Ex.~4]{CSOA} and \cite[Thm.~2.3]{CCO}),
	it is often difficult to tell whether or not a given matrix is UECSM.
	For instance, one of the following matrices is UECSM, but it is impossible to
	determine which one based upon existing methods in the literature:
	\begin{equation}\label{eq-Puzzle}
		\begin{pmatrix}
			 5 & 1 & 1 & 3 \\
			 1 & 1 & 1 & -1 \\
			 1 & -3 & 5 & -1 \\
			 -1 & -1 & -1 & 1
		\end{pmatrix}
		,\qquad
		\begin{pmatrix}
			 5 & -1 & 3 & 3 \\
			 1 & 3 & -1 & -1 \\
			 1 & -1 & 3 & -1 \\
			 -1 & 1 & -3 & 1
		\end{pmatrix}.
	\end{equation}
	On the other hand, the method which we introduce
	here easily dispatches this particular problem (see Section \ref{SectionComparison}).
	
	Although ad-hoc methods sometimes suffice for specific examples 
	(e.g., \cite[Ex.~5]{CSOA} \cite[Ex.~1, Thm.~4]{SNCSO}),
	the first general approach was due to J.~Tener \cite{Tener}
	who developed a procedure (\texttt{UECSMTest}), based upon the diagonalization 
	of the selfadjoint components $A$ and $B$ in the Cartesian decomposition $T = A + iB$, 
	by which a given matrix could be tested.  More recently, L.~Balayan and the first author 
	developed another procedure (\texttt{StrongAngleTest}), 
	based upon a careful analysis of the eigenstructure of $T$ itself \cite{Balayan}.  
	In this note, we pursue a different approach, based upon the diagonalization
	of $T^*T$ and $TT^*$.  It turns out that this method has several advantages over its counterparts
	(see Section \ref{SectionComparison}).
	
	Before discussing our main result, we require a few preliminary definitions.
	Recall that the singular values of a matrix $T \in M_n(\C)$
	are defined to be the eigenvalues of the positive matrix $|T| = \sqrt{T^*T}$, the 
	so-called \emph{modulus} of $T$.  
	We also remark that $T^*T$ and $TT^*$ share the same eigenvalues
	\cite[Pr.~101]{Halmos}.

	\begin{Theorem}\label{TheoremMain}
		If $T \in M_n(\C)$ has distinct singular values,
		\begin{enumerate}\addtolength{\itemsep}{0.25\baselineskip}
			\item $u_1,u_2,\ldots,u_n$ are unit eigenvectors of $T^*T$ corresponding 
				to the eigenvalues $\lambda_1,\lambda_2,\ldots,\lambda_n$, respectively,
			\item $v_1,v_2,\ldots,v_n$ are unit eigenvectors of $TT^*$ corresponding 
				to the eigenvalues $\lambda_1,\lambda_2,\ldots,\lambda_n$, respectively,
		\end{enumerate}
		then $T$ is UECSM if and only if
		\begin{align}
			|\inner{u_i,v_j}| &= |\inner{u_j,v_i}| , \label{eq-Magnitude} \\
			\inner{u_i,v_j}\inner{u_j,v_k}\inner{u_k,v_i} 
			&= \inner{u_i,v_k} \inner{u_k,v_j} \inner{u_j,v_i} 	,\label{eq-Cocycle}
		\end{align}
		holds for $1 \leq i<j<k \leq n$.
	\end{Theorem}	

	The procedure suggested by the preceding theorem can easily be implemented
	in \texttt{Mathematica}.  We refer to this procedure as \texttt{ModulusTest}.

	The structure of this paper is as follows.  
	The proof of Theorem \ref{TheoremMain} is the subject of Section \ref{SectionProof}.
	Section \ref{SectionExamples} contains a number of instructive examples.
	In Section \ref{SectionComparison} we compare \texttt{ModulusTest}
	to the procedures \texttt{UECSMTest} \cite{Tener} and \texttt{StrongAngleTest} \cite{Balayan}.
	We highlight several advantages of our approach over these other methods.	
	In Section \ref{SectionVolterra} we discuss applications of our results
	to compact operators.  As an illustration, we reveal a ``hidden symmetry'' of the Volterra integration operator.

\begin{Example}\label{ExampleOMG}
Before we proceed, we list several matrices which are UECSM and
their corresponding complex symmetric matrices.  These matrices
were tested by \texttt{ModulusTest} and the unitary equivalences
exhibited using the procedures outlined in Section \ref{SectionExamples}.
In particular, we have selected relatively
simple matrices which enjoy no apparent ``symmetry''
whatsoever.  The symbol $\cong$ denotes unitary equivalence.
\begin{align*}
\begin{pmatrix}
 4 & 1 & 1 \\
 4 & 5 & 7 \\
 4 & 7 & 5
\end{pmatrix}
&\cong \micro{6}
\begin{pmatrix}
 8-\sqrt{\frac{57}{2}} & -i \sqrt{\frac{1539}{481}-\frac{36 \sqrt{114}}{481}} & -3 i \sqrt{\frac{1}{962} \left(139+8 \sqrt{114}\right)} \\
 -i \sqrt{\frac{1539}{481}-\frac{36 \sqrt{114}}{481}} & \frac{7}{37} \left(22+\sqrt{114}\right) & \frac{1}{37} \sqrt{41553+3616 \sqrt{114}} \\
 -3 i \sqrt{\frac{1}{962} \left(139+8 \sqrt{114}\right)} & \frac{1}{37} \sqrt{41553+3616 \sqrt{114}} & \frac{1}{74} \left(136+23 \sqrt{114}\right)
\end{pmatrix}
\\
\begin{pmatrix}
 5 & 2 & 2 \\
 7 & 0 & 0 \\
 7 & 0 & 0
\end{pmatrix}
&\cong\micro{5}
\begin{pmatrix}
 \frac{1}{2} \left(5-\sqrt{187}\right) & -5 i \sqrt{\frac{561+5 \sqrt{187}}{1658}} & -i \sqrt{\frac{3350}{829}-\frac{125 \sqrt{187}}{1658}} \\
 -5 i \sqrt{\frac{561+5 \sqrt{187}}{1658}} & \frac{1}{829} \left(1870+293 \sqrt{187}\right) & \frac{9}{829} \sqrt{\frac{1}{2} \left(173723+7075 \sqrt{187}\right)} \\
 -i \sqrt{\frac{3350}{829}-\frac{125 \sqrt{187}}{1658}} & \frac{9}{829} \sqrt{\frac{1}{2} \left(173723+7075 \sqrt{187}\right)} & \frac{81}{-5+3 \sqrt{187}}
\end{pmatrix}
\\
\begin{pmatrix}
 9 & 8 & 9 \\
 0 & 7 & 0 \\
 0 & 0 & 7
\end{pmatrix}
&\cong\micro{8}
\begin{pmatrix}
 8-\frac{\sqrt{149}}{2} & \frac{9}{2} i \sqrt{\frac{16837+64 \sqrt{149}}{13093}} & i \sqrt{\frac{133672}{13093}-\frac{1296 \sqrt{149}}{13093}} \\
 \frac{9}{2} i \sqrt{\frac{16837+64 \sqrt{149}}{13093}} & \frac{207440+9477 \sqrt{149}}{26186} & \frac{18 \sqrt{3978002+82324 \sqrt{149}}}{13093} \\
 i \sqrt{\frac{133672}{13093}-\frac{1296 \sqrt{149}}{13093}} & \frac{18 \sqrt{3978002+82324 \sqrt{149}}}{13093} & \frac{92675+1808 \sqrt{149}}{13093}
\end{pmatrix} 
\end{align*}
\end{Example}

\section{Proof of Theorem \ref{TheoremMain}}\label{SectionProof}

\subsection{Preliminary lemmas}	
	Recall that a conjugation $C$ on $\C^n$ is a conjugate-linear involution (i.e., $C^2 = I$)
	which is also isometric (i.e., $\inner{Cx,Cy} = \inner{y,x}$ for all $x,y \in \C^n$). 
	It is easy to see that each conjugation $C$ on $\C^n$ is of the form $C = SJ$
	where $S$ is a complex symmetric unitary matrix and $J$ is the canonical conjugation 
	\begin{equation}\label{eq-Canonical}
		J(z_1,z_2,\ldots,z_n) = (\overline{z_1}, \overline{z_2}, \ldots, \overline{z_n})
	\end{equation}
	on $\C^n$.  The relevance of conjugations to our endeavor lies in the following lemma.

	\begin{Lemma}\label{LemmaC}
		$T \in M_n(\C)$ is UECSM if and only if there exists a conjugation $C$ on $\C^n$
		such that $T = CT^*C$.
	\end{Lemma}

	\begin{proof}
			Suppose that $T = CT^*C$ for some conjugation $C$ on $\C^n$.
			By \cite[Lem.~1]{CSOA} there exists 
			an orthonormal basis $e_1,e_2,\ldots,e_n$ such that $Ce_i = e_i$
			for $i = 1,2,\ldots,n$.  Let $Q = (e_1 | e_2| \cdots |e_n)$ be the
			unitary matrix whose columns are these basis vectors.  The
			matrix $S = Q^*TQ$ is complex symmetric since the $ij$th 
			entry $[S]_{ij}$ of $S$ satisfies
			$[S]_{ij} = \inner{T e_j,e_i} = \inner{CT^*Ce_j,e_i} 
			= \inner{e_i, T^*e_j} = \inner{Te_i,e_j} = [S]_{ji}.$
	\end{proof}
	
	Our next result shows that,
	under the hypotheses of Theorem \ref{TheoremMain},
	$T$ is UECSM if and only if there is a conjugation intertwining $T^*T$ and $TT^*$.
	
	\begin{Lemma}\label{LemmaCTTC}
		If $C$ is a conjugation on $\C^n$ and $T\in M_n(\C)$ has distinct singular values, then
		\begin{equation}\label{eq-CTTC}
			T = CT^*C \quad \Leftrightarrow \quad T^*T = C(TT^*)C.
		\end{equation}
	\end{Lemma}

	\begin{proof}
		The $(\Rightarrow)$ implication of \eqref{eq-CTTC} follows immediately,
		regardless of any hypotheses on the singular values of $T$.  
		The implication
		$(\Leftarrow)$ is considerably more involved.
		Suppose that $T^*T = CTT^*C$.
		Write $T = U(T^*T)^{\frac{1}{2}}$ where $U$ is unitary and observe that
		$TT^* = UT^*TU^*$ whence $UT^*T = TT^*U$.  It follows that
		$UT^*T = CT^*TCU$ which implies that
		\begin{equation}\label{eq-CUCommutes}
			CU(T^*T) = (T^*T)CU.
		\end{equation}
		Let $e_1,e_2,\ldots,e_n$ denote unit eigenvectors of $T^*T$ corresponding to the 
		(necessarily non-negative) eigenvalues $\lambda_1,\lambda_2,\ldots,\lambda_n$ of $T^*T$.

		In light of \eqref{eq-CUCommutes}, we see that $T^*T e_i = \lambda_i e_i$ if and only if 
		$(T^*T)(CUe_i) = \lambda_i (CUe_i)$.  In other words, the conjugate-linear operator $CU$
		maps each eigenspace of $T^*T$ into itself.  Since $CU$ is isometric and since the eigenspaces 
		of $T^*T$ are one-dimensional, it follows that $CUe_i = \zeta_i^2 e_i$ for some unimodular constants
		$\zeta_1,\zeta_2,\ldots,\zeta_n$.
		Using the fact that $C$ is conjugate-linear we find that the unit vectors 
		$w_i = \zeta_i e_i$ satisfy $CUw_i = w_i$ and $T^*Tw_i = \lambda_i w_i$.

		We claim that the conjugate-linear operator $K = CU$ is a conjugation on $\C^n$.
		Indeed, since $U$ is unitary and $C$ is a conjugation it is clear that $K$ is isometric.  
		Moreover, since $K^2w_i = CUCUw_i = CUw_i = w_i$ for $i = 1,2,\ldots,n$ it follows that $K^2 = I$
		whence $K$ is a conjugation.  By \eqref{eq-CUCommutes} it follows that
		$K(T^*T)K = T^*T$ whence $J|T|J = |T|$ (since $|T| = p(T^*T)$
		for some polynomial $p(x) \in \R[x]$).

		Putting this all together, we find that
		$T = CK|T|$ where $K$ is a conjugation that commutes with $|T|$.  In particular,
		the unitary matrix $U$ factors as $U = CK$ and satisfies $U^* = KC$.  We therefore conclude that
		$T = CK|T| = C|T|K = C(|T|KC)C = C(|T|U^*)C = CT^*C$.
	\end{proof}

	We remark that the implication $(\Leftarrow)$ of Lemma \ref{LemmaCTTC} is false
	if one drops the hypothesis that the singular values of $T$ are distinct.  For instance, let $T$
	be unitary matrix which is not complex symmetric (i.e., $T \neq JT^*J$ where $J$ denotes the canonical
	conjugation \eqref{eq-Canonical} on $\C^n$).  In this case, $T^*T = I = TT^*$ (i.e., all of the singular vales of $T$ are $1$)
	and hence the condition on the right-hand side of \eqref{eq-CTTC} obviously holds.  
	On the other hand, $T \neq JT^*J$ by hypothesis.
	
	From here on, we maintain the notation and conventions of Theorem \ref{TheoremMain},
	namely that $u_1,u_2,\ldots,u_n$ are unit eigenvectors of $T^*T$
	and $v_1,v_2,\ldots,v_n$ are unit eigenvectors of $TT^*$ corresponding 
	to the eigenvalues $\lambda_1,\lambda_2,\ldots,\lambda_n$, respectively.

	\begin{Lemma}\label{LemmaUnimodular}
		If $C$ is a conjugation on $\C^n$ and $T\in M_n(\C)$ has distinct singular values,
		then $T^*T = CTT^*C$ if and only if $Cu_i = \alpha_i v_i$ for 
		some unimodular constants $\alpha_1,\alpha_2,\ldots, \alpha_n$.
	\end{Lemma}

	\begin{proof}
		For the forward implication, observe that
		$\lambda_i u_i = T^*Tu_i = CTT^*Cu_i$ whence $TT^*(Cu_i) = \lambda_i (Cu_i)$.  Since the eigenspaces
		of $TT^*$ are one-dimensional and $C$ is isometric, it follows that $Cu_i = \alpha_i v_i$ for some unimodular constants
		$\alpha_1,\alpha_2,\ldots, \alpha_n$.

		On the other hand, suppose that there exist unimodular constants $\alpha_1,\alpha_2,\ldots, \alpha_n$
		such that $Cu_i = \alpha_i v_i$ for $i = 1,2,\ldots,n$.  Since $C$ is a conjugation, it follows that
		$Cv_i = \alpha_i u_i$ for $i = 1,2,\ldots,n$.  It follows that
		$CTT^*Cu_i = CTT^*\alpha_i v_i = \overline{\alpha_i}CTT^* v_i = 
			\overline{\alpha_i}\lambda_i Cv_i = \overline{\alpha_i}\alpha_i\lambda_i u_i = \lambda_i u_i$
		for $i = 1,2,\ldots,n$.  Since the linear operators $CTT^*C$ and $T^*T$ agree on the 
		orthonormal basis $u_1,u_2,\ldots,u_n$, we conclude that $T^*T = CTT^*C$.
	\end{proof}

	\begin{Lemma}\label{LemmaAlpha}
		There exists a conjugation $C$ and unimodular constants $\alpha_1,\alpha_2,\ldots, \alpha_n$		
		such that $Cu_i = \alpha_i v_i$ for $i=1,2,\ldots,n$ if and only if 
		\begin{equation}\label{eq-AlphaCondition}
			\inner{u_i,v_j } = \alpha_j \overline{\alpha_i} \inner{u_j, v_i}
		\end{equation}
		holds for $1 \leq i,j \leq n$.
	\end{Lemma}

	\begin{proof}
		For the forward implication, simply note that if $Cu_i = \alpha_i v_i$ for 
		$i=1,2,\ldots,n$, then \eqref{eq-AlphaCondition}
		follows immediately from the fact that $C$ is isometric and conjugate-linear.
		Conversely, suppose that \eqref{eq-AlphaCondition} holds for $1 \leq i,j \leq n$.
		We claim that the definition $Cu_i = \alpha_i v_i$ for $1 \leq i \leq n$ extends by conjugate-linearity
		to a conjugation on all of $\C^n$.  Since $u_1,u_2,\ldots,u_n$ and $v_1,v_2,\ldots,v_n$ are orthonormal
		bases of $\C^n$ and since the constants $\alpha_1,\alpha_2,\ldots, \alpha_n$ are unimodular, 
		it follows that $C$ is isometric.  It therefore suffices to prove that 
		$C^2 = I$.  To this end, we need only show that $Cv_i = \alpha_i u_i$ for $1 \leq i \leq n$.
		This follows from a straightforward computation:
		\begin{align*}
			Cv_i
			&= C\left( \sum_{j=1}^n \inner{v_i,u_j}u_j \right) 
			=  \sum_{j=1}^n \inner{u_j,v_i}Cu_j 
			=  \sum_{j=1}^n \inner{u_j,v_i}\alpha_j v_j \\
			&=  \sum_{j=1}^n \alpha_i \overline{\alpha_j} \inner{u_i,v_j}\alpha_j v_j 
			=  \alpha_i \sum_{j=1}^n  \inner{u_i,v_j} v_j 
			=  \alpha_i u_i.
		\end{align*}
		Thus $C$ is a conjugation on $\C^n$, as desired.
	\end{proof}

	We can interpret the condition \eqref{eq-AlphaCondition} in terms of matrices.
	Let $U = (u_1 | u_2 | \cdots | u_n)$ and $V = (v_1 | v_2 | \cdots |v_n)$ denote
	the $n \times n$ unitary matrices whose columns are the orthonormal bases $u_1,u_2,\ldots,u_n$
	and $v_1,v_2,\ldots,v_n$, respectively.  Now 	
	observe that \eqref{eq-AlphaCondition} is equivalent to asserting that
	\begin{equation}\label{eq-UVA}
		(V^*U)^t = A^*(V^*U)A
	\end{equation}
	holds where $A = \operatorname{diag}(\alpha_1,\alpha_2,\ldots,\alpha_n)$ denotes 
	the diagonal unitary matrix having the unimodular constants $\alpha_1,\alpha_2,\ldots,\alpha_n$
	along the main diagonal.

	Putting Lemmas \ref{LemmaCTTC}, \ref{LemmaUnimodular}, and \ref{LemmaAlpha} together,
	we obtain the following important lemma.

	\begin{Lemma}\label{LemmaMain}
		There exist unimodular constants $\alpha_1,\alpha_2,\ldots, \alpha_n$ such that \eqref{eq-AlphaCondition} holds
		if and only if $T$ is UECSM.
	\end{Lemma}

	With these preliminaries in hand, we are now ready to complete the proof of Theorem \ref{TheoremMain}.

\subsection{Proof of the implication $(\Rightarrow)$}
	Suppose that $T$ is UECSM.  By Lemma \ref{LemmaMain}, there exist unimodular constants
	$\alpha_1,\alpha_2,\ldots,\alpha_n$ so that \eqref{eq-AlphaCondition} holds for $1 \leq i,j \leq n$.
	The desired conditions \eqref{eq-Magnitude} and \eqref{eq-Cocycle} from the statement of 
	Theorem \ref{TheoremMain} then follow immediately.
	
\subsection{Proof of the implication $(\Leftarrow)$}
	The proof that conditions \eqref{eq-Magnitude} and \eqref{eq-Cocycle} are sufficient for 
	$T$ to be UECSM is somewhat more complicated.  Fortunately, the proof of
	\cite[Thm.~2]{Balayan} goes through, \emph{mutatis mutandis}, and we refer the reader
	there for the details.  We sketch the main idea below.
	
	Suppose that $\inner{u_j,v_i} \neq 0$ for $1 \leq i,j \leq n$ (the proof of \cite[Thm.~2]{Balayan}
	explains how to get around this restriction) and observe that \eqref{eq-Magnitude} ensures that the constants 
	\begin{equation*}
		\beta_{ij} = \frac{ \inner{u_i,v_j} }{ \inner{ u_j, v_i} }
	\end{equation*}
	are unimodular.  The condition \eqref{eq-Cocycle} then
	implies that $\beta_{ij} \beta_{jk} =\beta_{ik}$, from which it follows that the unimodular constants
	$\alpha_i = \beta_{1i}$ satisfy \eqref{eq-AlphaCondition}.  
	We therefore conclude that $T$ is UECSM by Lemma \ref{LemmaMain}. \qed

\section{Examples and computations}\label{SectionExamples}

	Before considering several examples, let us first remark that Theorem \ref{TheoremMain}
	is constructive.
	Maintaining the notation and conventions established in the proof of Theorem \ref{TheoremMain}, define 
	the unitary matrices $U$, $V$, and $A$ as in \eqref{eq-UVA}.
	Let $s_1,s_2,\ldots,s_n$ denote the standard basis of $\C^n$ and let
	$J$ denote the canonical conjugation \eqref{eq-Canonical} on $\C^n$.
	In particular, observe that $Js_i = s_i$ for $i = 1,2,\ldots, n$.
	The proof of Theorem \ref{TheoremMain} tells us that if $T$ satisfies \eqref{eq-Magnitude}
	and \eqref{eq-Cocycle} (e.g., ``$T$ passes \texttt{ModulusTest}''), then there exist unimodular constants
	$\alpha_1,\alpha_2,\ldots,\alpha_n$ such that $Cu_i = \alpha_i v_i$ for $i=1,2,\ldots,n$.
	Letting $A = \operatorname{diag}(\alpha_1, \alpha_2, \ldots, \alpha_n)$ we see that
	\begin{align*}
		VAU^t J u_i
		&= VAJU^*u_i 
		= VAJs_i \\
		&= VAs_i 
		= \alpha_i Vs_i \\
		&= \alpha_i v_i.
	\end{align*}
	Thus the conjugate-linear operators $C$ and $(VAU^t)J$ agree on the orthonormal basis
	$u_1,u_2,\ldots,u_n$ whence they agree on all of $\C^n$.  Although it is not immediately obvious,
	the unitary matrix $S = VAU^t$ is complex symmetric.  Indeed, the condition
	$S = S^t$ is equivalent to \eqref{eq-UVA}.
	Once the conjugation $C = SJ$ has been obtained it is a simple matter of
	finding an orthonormal basis with respect to which $T$ has a complex symmetric matrix
	representation (see Lemma \ref{LemmaC}).  To find such a basis, observe that 
	since $S = CJ$ is a $C$-symmetric unitary operator, each of its eigenspaces are 
	fixed by $C$ \cite[Lem.~8.3]{CCO}.  Some of the following examples illustrate this construction.

 	\begin{Example}\label{Example2x2}
		Although at this point many different proofs of the fact that every $2 \times 2$ matrix is UECSM exist
		(see \cite[Cor.~3]{Balayan}, \cite[Cor.~3.3]{Chevrot}, \cite[Ex.~6]{CSOA}, \cite{OPHUEMT},
		\cite[Cor.~1]{SNCSO}, \cite[Cor.~3]{Tener}), for the sake of illustration we give yet another.  	
		By Schur's Theorem on unitary triangularization,
		we need only consider only upper triangular $2 \times 2$ matrices.
		If $T$ is such a matrix and has repeated eigenvalues, then upon subtracting a multiple of the identity
		we may assume that
		\begin{equation*}
			T = \minimatrix{0}{a}{0}{0}.
		\end{equation*}
		A routine computation now shows that $T = UAU^*$ where 
		\begin{equation*}
			A = \minimatrix{ \frac{a}{2} }{ \frac{ia}{2} }{ \frac{ia}{2} }{ -\frac{a}{2} }, \qquad
			U = \minimatrix{ \frac{1}{\sqrt{2}} }{ \frac{-i}{\sqrt{2}} }{ \frac{1}{\sqrt{2}} }{ \frac{i}{\sqrt{2}} }.
		\end{equation*}
		Thus it suffices to consider the case where $T$ has distinct eigenvalues.  Upon
		subtracting a multiple of the identity and then scaling, we may assume that 
		\begin{equation*}
			T = \minimatrix{1}{a}{0}{0}.
		\end{equation*}
		Moreover, we may also assume that $a \geq 0$ since this may be obtained
		by conjugating $T$ by an appropriate diagonal unitary matrix.  Thus we have
		\begin{equation*}
			T^*T = \minimatrix{1}{a}{a}{a^2}, \quad TT^* = \minimatrix{1+a^2}{0}{0}{0}.
		\end{equation*}
		The eigenvalues of $T^*T$ and $TT^*$ are $\lambda_1 = 1+a^2$ and $\lambda_2 = 0$
		and corresponding unit eigenvectors are	
		\begin{equation*}
			u_1 = \twovector{ \frac{1}{ \sqrt{1+a^2}} }{ \frac{a}{\sqrt{1+a^2}} }, \quad
			u_2 = \twovector{ \frac{-a}{\sqrt{1+a^2}} }{ \frac{1}{\sqrt{1+a^2}}}, \quad
			v_1 = \twovector{1}{0}, \quad
			v_2 = \twovector{0}{1}.
		\end{equation*}
		Let us first consider the condition \eqref{eq-Magnitude} of the procedure \texttt{ModulusTest}.
		For $i = j$ it holds trivially and for $i \neq j$ we have
		\begin{equation*}
			|\inner{u_1,v_2}| = \frac{a}{\sqrt{1+a^2}} = | \inner{ u_2,v_1} |. 
		\end{equation*}
		Now let us consider the second condition \eqref{eq-Cocycle}.
		Since $n = 2$, at least two of $i,j,k$ must be equal whence \eqref{eq-Cocycle} holds trivially.
		By Theorem \ref{TheoremMain}, it follows that $T$ is UECSM.  
		
		Let us now explicitly construct a complex symmetric matrix which $T$ is unitarily
		equivalent to.  Since the equation
		\begin{equation*}
			\frac{a}{ \sqrt{1+a^2}} = \inner{u_1,v_2} = \overline{\alpha_1} \alpha_2 \inner{u_2,v_1} =
			\overline{\alpha_1} \alpha_2 \frac{-a}{ \sqrt{1+a^2}}
		\end{equation*}
		is satisfied by $\alpha = 1$ and $\alpha_2 = -1$, we let
		\begin{equation*}
			S=
			\underbrace{\left( 
				\begin{array}{c|c}
					1&0\\
					0&1
				\end{array}
			\right)}_{V}
			\underbrace{\left( 
				\begin{array}{cc}
					1&0\\
					0&-1
				\end{array}
			\right) }_{A}
			\underbrace{
			\left( 
				\begin{array}{cc}
					\frac{1}{ \sqrt{1+a^2}}&\frac{a}{\sqrt{1+a^2}}\\
					\hline
					\frac{-a}{\sqrt{1+a^2}} &\frac{1}{\sqrt{1+a^2}}
				\end{array}
			\right) }_{U^t}
			=  \minimatrix{\frac{1}{\sqrt{1+a^2}}}{\frac{a}{\sqrt{1+a^2}}}{\frac{a}{\sqrt{1+a^2}}}{-\frac{1}{\sqrt{1+a^2}}}
		\end{equation*}
		and note that the conjugation $C = SJ$ satisfies $T = CT^*C$.  An orthonormal basis $e_1,e_2$ of $\C^2$
		whose elements are fixed by $C$ is given by
		\begin{equation*}
		e_1 = \twovector{  \frac{1 - \sqrt{1+a^2}}{\sqrt{2 + 2a^2 - 2 \sqrt{1+a^2}}}   }{  \frac{a}{\sqrt{2 + 2a^2 - 2 \sqrt{1+a^2}}}   }\qquad
		e_2 = \twovector{ \frac{-ia}{\sqrt{2 + 2a^2 - 2 \sqrt{1+a^2}}}  }{   \frac{i(1 - \sqrt{1+a^2})}{\sqrt{2 + 2a^2 - 2 \sqrt{1+a^2}}}  }.
		\end{equation*}
		Note that these are certain normalized eigenvectors of $S$, corresponding to the eigenvalues $1$ and $-1$, respectively,
		whose phases are selected so that $Ce_1 = SJe_1 = Se_1 = e_1$ and $Ce_2 = SJe_2 = S(-e_2) = -Se_2 = e_2$.
		Letting $Q= (e_1|e_2)$ denote the unitary matrix whose columns are $e_1$ and $e_2$, we find that
		\begin{equation*}
			Q^*TQ = \minimatrix{ \frac{1}{2}(1 - \sqrt{1+a^2}) }{ \frac{ia}{2} }{\frac{ia}{2}}{ \frac{1}{2}(1 + \sqrt{1+a^2})}.
		\end{equation*}
		As predicted by Lemma \ref{LemmaC}, this matrix is complex symmetric.
	\end{Example}

	The following simple example was first considered, using ad-hoc methods, in \cite[Ex.~1]{SNCSO}.
	Note that the procedure \texttt{StrongAngleTest} of \cite{Balayan} cannot be applied to this matrix 
	due to the repeated eigenvalue $0$.
		
	\begin{Example}\label{ExampleAB}
		Suppose that $ab \neq 0$ and $|a| \neq |b|$.  In this case, the singular values of 
		\begin{equation*}
			T = \megamatrix{0}{a}{0}{0}{0}{b}{0}{0}{0}
		\end{equation*}
		are distinct.  Normalized eigenvectors $u_1,u_2,u_3$ of $T^*T$ and 
		$v_1,v_2,v_3$ of $TT^*$ corresponding to the eigenvalues $0,|a|^2,|b|^2$, 
		respectively are given by
		\begin{equation*}
			u_1 = v_3 = \threevector{1}{0}{0},\qquad
			u_2 = v_1 = \threevector{0}{1}{0},\qquad
			u_3 = v_2 = \threevector{0}{0}{1}.
		\end{equation*}
		Since $\inner{u_1,v_2} = 0$ and $\inner{u_2,v_1} = 1$, condition \eqref{eq-Magnitude}
		fails from which we conclude that $T$ is not UECSM.  

		On the other hand, if either $a=0$ or $b=0$, then $T$ is the direct sum of a $1\times 1$ with 
		a $2 \times 2$ matrix whence $T$ is UECSM by Example \ref{Example2x2}.  Moreover, if $|a| = |b|$,
		then $T$ is unitarily equivalent to a Toeplitz matrix and thus UECSM by \cite[Sect.~2.2]{CCO}.  
	\end{Example}
	


	\begin{Example}\label{ExampleNasty}
		We claim that the lower-triangular matrix 
		\begin{equation*}
			T = 
			\begin{pmatrix}
				 0 & 0 & 0 \\
				 1 & 2 & 0 \\
				 1 & 0 & 2
			\end{pmatrix}
		\end{equation*}
		is UECSM.  Normalized eigenvectors $u_1,u_2,u_3$ of $T^*T$ and $v_1,v_2,v_3$ of $TT^*$ corresponding
		to the eigenvalues $\lambda_1 = 0$, $\lambda_2 = 4$, and $\lambda_3 = 6$ are given by
		\begin{equation*}
			u_1 = \threevector{ \frac{1}{\sqrt{3}} }{ \frac{1}{\sqrt{3}} }{ \frac{1}{\sqrt{3}} }, \qquad
			u_2 = \threevector{0}{- \frac{1}{\sqrt{2}}}{ \frac{1}{\sqrt{2}}}, \qquad
			u_3 = \threevector{ - \frac{4}{\sqrt{6} } }{ \frac{1}{\sqrt{6} } }{ \frac{1}{\sqrt{6}}},
		\end{equation*}
		and
		\begin{equation*}
			v_1 = \threevector{0}{ \frac{1}{\sqrt{2} } }{ \frac{1}{\sqrt{2} } }, \qquad
			v_2 = \threevector{0}{ -\frac{1}{\sqrt{2} } }{ \frac{1}{\sqrt{2} } }, \qquad
			v_3 = \threevector{1}{0}{0},
		\end{equation*}
		respectively.  Since
		\begin{align*}
			\inner{u_1,v_2} & = \inner{u_2,v_1} = 0, \\
			\inner{u_2,v_3} &=\inner{u_3,v_2} =0,\\
			\inner{u_3,v_1} &= \inner{u_1,v_3} =  \tfrac{1}{\sqrt{3}},
		\end{align*}
		conditions \eqref{eq-Magnitude} and \eqref{eq-Cocycle} are obviously satisfied.
		By Theorem \ref{TheoremMain}, we conclude that $T$ is UECSM.  Let us now
		construct a complex symmetric matrix which $T$ is unitarily equivalent to.

		By inspection, we find that
		$\alpha_1 = \alpha_2 = \alpha_3 = 1$ is a solution to \eqref{eq-AlphaCondition}.
		Maintaining the notation established at the beginning of this section, we observe that the matrix
		\begin{align*}
			S 
			&=
			\underbrace{
			\left(
			\begin{array}{c|c|c}
				 0 & 0 & 1 \\
				 \frac{1}{\sqrt{2}} & -\frac{1}{\sqrt{2}} & 0 \\
				 \frac{1}{\sqrt{2}} & \frac{1}{\sqrt{2}} & 0
			\end{array}
			\right)
			}_V 
			\underbrace{\megamatrix{1}{0}{0}{0}{1}{0}{0}{0}{1}}_A
			\underbrace{
			\left(
			\begin{array}{ccc}
				\frac{1}{\sqrt{3}} & \frac{1}{\sqrt{3}} & \frac{1}{\sqrt{3}} \\[3pt]
				\hline
				0 & -\frac{1}{\sqrt{2}} & \frac{1}{\sqrt{2}} \\[3pt]
				\hline
				-\frac{4}{\sqrt{6}} & \frac{1}{\sqrt{6}} & \frac{1}{\sqrt{6}}
			\end{array}
			\right)
			}_{U^t}
			\\
			&=
			\begin{pmatrix}
				 -\frac{4}{\sqrt{6}}  & \frac{1}{\sqrt{6}} & \frac{1}{\sqrt{6}} \\
				 \frac{1}{\sqrt{6}} & \frac{1}{2}+\frac{1}{\sqrt{6}} & -\frac{1}{2}+\frac{1}{\sqrt{6}} \\
				 \frac{1}{\sqrt{6}} & -\frac{1}{2}+\frac{1}{\sqrt{6}} & \frac{1}{2}+\frac{1}{\sqrt{6}}
			\end{pmatrix}
		\end{align*}
		is symmetric and unitary.   We then find an orthonormal basis $e_1,e_2,e_3$
		whose elements are fixed by the conjugation $C = SJ$.  Following
		Lemma \ref{LemmaC}, we encode one such example as the columns of the unitary matrix
		\begin{equation*}\small
			Q =
			\left(
			\begin{array}{c|c|c}
				 -i \sqrt{\frac{1}{2}+\frac{1}{\sqrt{6}}} & \frac{1}{5} \sqrt{11-4 \sqrt{6}} & \frac{1}{\sqrt{2 \left(9+\sqrt{6}\right)}} \\[5pt]
				 \frac{i}{2 \sqrt{3+\sqrt{6}}} & 0 & \frac{1}{2} \sqrt{3+\sqrt{\frac{2}{3}}} \\
				 \frac{i}{2 \sqrt{3+\sqrt{6}}} & \frac{1}{5} \left(\sqrt{2}+2 \sqrt{3}\right) & -\frac{1}{10} \sqrt{19-23 \sqrt{\frac{2}{3}}}
			\end{array}
			\right)
		\end{equation*}
		and note that $Q^*TQ$ is complex symmetric:
		\begin{equation*}\small
			\begin{pmatrix}
				 1-\sqrt{\frac{3}{2}} & -\frac{1}{5} i \sqrt{9-\sqrt{6}} & -\frac{1}{5} i \sqrt{\frac{7}{2}+\sqrt{6}} \\[5pt]
				 -\frac{1}{5} i \sqrt{9-\sqrt{6}} & \frac{1}{25} \left(26+11 \sqrt{6}\right) & \frac{1}{25} \sqrt{123-47 \sqrt{6}} \\[5pt]
				 -\frac{1}{5} i \sqrt{\frac{7}{2}+\sqrt{6}} & \frac{1}{25} \sqrt{123-47 \sqrt{6}} & \frac{1}{50} \left(98+3 \sqrt{6}\right)
			\end{pmatrix}.
		\end{equation*}
		Independent confirmation that $T$ is UECSM is obtained by noting that
		$T - 2I$ has rank one (every rank-one matrix is UECSM by \cite[Cor.~5]{SNCSO}).
\end{Example}

\section{Comparison with other methods}\label{SectionComparison}
	
	With the addition of \texttt{ModulusTest} there are now three general 
	procedures for determining whether a 
	matrix $T$ is UECSM.  Each has its own restrictions:	
	\begin{enumerate}\addtolength{\itemsep}{0.25\baselineskip}
		\item \texttt{ModulusTest} (this article) requires that $T$ has distinct singular values,
		\item \texttt{StrongAngleTest} \cite{Balayan} requires that $T$ has distinct eigenvalues,
		\item \texttt{UECSMTest} \cite{Tener} requires that the selfadjoint matrices $A,B$ in the Cartesian
			decomposition $T = A + iB$ (where $A=A^*$, $B = B^*$) both have distinct eigenvalues.
			However, this restriction can be removed in the $3 \times 3$ case.
	\end{enumerate}
	In this section, we compare \texttt{ModulusTest} to these other methods and point out
	several advantages of our procedure.

	Table \ref{TableA} provides a number of examples indicating that
	\texttt{ModulusTest} is not subsumed by the other two procedures
	mentioned above.  
	At this point we should also remark that the two matrices \eqref{eq-Puzzle} from the introduction
	are unitarily equivalent to constant multiples of the corresponding matrices in Table \ref{TableA}.
	In particular, the first matrix in \eqref{eq-Puzzle} is unitarily equivalent to
\begin{equation*}\micro{8}
\begin{pmatrix}
 \frac{2}{17} \left(23+16 \sqrt{2}\right) & \frac{4}{17} \sqrt{50-31 \sqrt{2}} & -2 i \sqrt{\frac{1}{17} \left(5+2 \sqrt{2}\right)} & -i \sqrt{\frac{48}{17}-\frac{8 \sqrt{2}}{17}} \\
 \frac{4}{17} \sqrt{50-31 \sqrt{2}} & \frac{2}{17} \left(45+\sqrt{2}\right) & -i \sqrt{\frac{48}{17}-\frac{8 \sqrt{2}}{17}} & 2 i \sqrt{\frac{1}{17} \left(5+2 \sqrt{2}\right)} \\
 -2 i \sqrt{\frac{1}{17} \left(5+2 \sqrt{2}\right)} & -i \sqrt{\frac{48}{17}-\frac{8 \sqrt{2}}{17}} & 2 & 0 \\
 -i \sqrt{\frac{48}{17}-\frac{8 \sqrt{2}}{17}} & 2 i \sqrt{\frac{1}{17} \left(5+2 \sqrt{2}\right)} & 0 & 2-2 \sqrt{2}
\end{pmatrix}.
\end{equation*}

	\begin{table}
		\begin{equation*}\small
			\begin{array}{|c|c|c|c|c|c|}
				\hline
				T & \sigma(T^*T) & \sigma(T) & \sigma(A) & \sigma(B) & \text{UECSM?} \\
				\hline
				\begin{pmatrix}
					1 & 0 & 1 & 0 \\
					0 & 1 & 0 & 0  \\
					0 & 0 & 1 & 0 \\
					0 &1 & 0 & 0 \\				
				\end{pmatrix}
				& 0,2,\frac{3 \pm \sqrt{5}}{2}
				& 0,1,1,1
				& \frac{1}{2}, \frac{3}{2}, \frac{1\pm\sqrt{2}}{2}
				&  - \frac{1}{2}, - \frac{1}{2}, \frac{1}{2} \frac{1}{2} 
				& \textsc{yes}\\
				\hline
				\begin{pmatrix}
					1 & 0 & 1 & 0 \\
					0 & 1 & 0 & 0  \\
					0 & 0 & 1 & 0 \\
					0 &0 & 1 & 0 \\
				\end{pmatrix}
				& 0,1, 2 \pm \sqrt{2}
				& 0,1,1,1
				& \text{distinct}
				&0,0,\pm\frac{\sqrt{2}}{2}
				& \textsc{no}
				\\
				\hline
			\end{array}
		\end{equation*}
		
		\caption{\footnotesize Matrices which satisfy the hypotheses of \texttt{ModulusTest}
		but not those of \texttt{UECSMTest} or \texttt{StrongAngleTest} (the notation $\sigma( \cdot )$ 
		denotes the spectrum of a matrix).  Whether or not
		these matrices are UECSM can be determined by \texttt{ModulusTest}.  In the second row,
		the eigenvalues of $A$ are distinct but cannot be displayed exactly in the confines of the table.}
		\label{TableA}
	\end{table}

	One major advantage that \texttt{ModulusTest} has over its competitors is due to the 
	nonlinear nature of the map $X \mapsto X^*X$ on $M_n(\C)$.
	First note that the property of being UECSM is invariant under translation 
	$X \mapsto X+cI$ for $c \in \C$.  Next observe that if $T$ does not satisfy the hypotheses 
	of \texttt{UECSMTest} or \texttt{StrongAngleTest}, then neither does $T+ cI$ for any value of $c$.
	On the other hand, $T+cI$ will often satisfy the hypotheses of \texttt{ModulusTest}
	even if $T$ itself does not.	
	
%
%

	\begin{table}
		\begin{equation*}\footnotesize
			\begin{array}{|c|c|c|c|c|c|}
				\hline
				T & \sigma(T^*T) & \sigma(T) & \sigma(A) & \sigma(B) & \text{UECSM?} \\
				\hline
				\begin{pmatrix}
					1 &0 & 0 & 0 \\
					0 & 0 & 2 & 0  \\
					0 & 0 & 0 & 2 \\
					0 &0 & 0 & 0 \\
				\end{pmatrix}
					& 0,1,4,4
					&0,0,0,1
					&0,1,\pm\sqrt{2}
					&0,0,\pm \sqrt{2}
				& \textsc{yes}
				\\
				\hline
				\begin{pmatrix}
					1 &0 & 0 & 0 \\
					0 & 0 & 1 & 0  \\
					0 & 0 & 0 & 2 \\
					0 &0 & 0 & 0 \\
				\end{pmatrix} 
				&0,1,1,4
				&0,0,0,1
				&0,1,\pm \frac{\sqrt{5}}{2}
				& 0,0, \pm \frac{\sqrt{5}}{2}
				& \textsc{no} \\
				\hline
			\end{array}
		\end{equation*}
		
		\caption{\footnotesize Matrices which cannot be tested by \texttt{UECSMTest}, 
		\texttt{StrongAngleTest}, or \texttt{ModulusTest}. 
		However, \texttt{ModulusTest} \emph{does} apply to $T + I$ and hence
		\texttt{ModulusTest} can be used indirectly to test the original matrix $T$.}
		
		\label{TableB}
	\end{table}

	Table \ref{TableB} displays two matrices which do not satisfy the hypotheses
	of any of the three tests that we have available.  Nevertheless, the translation trick
	described above renders these matrices indirectly susceptible to \texttt{ModulusTest}.  
	For instance, the first matrix in Table \ref{TableB} is unitarily equivalent to
	\begin{equation*}\small
		\begin{pmatrix}
		1& 0 & 0 & 0 \\
		0& 0 & 0 & i \sqrt{2}  \\
		0& 0 & 0 & \sqrt{2}  \\
		 0& i \sqrt{2} & \sqrt{2} & 0
		\end{pmatrix}.
	\end{equation*}

	Rather than grind through the computational details, we can use simple 
	ad-hoc means to independently confirm the results listed in Table \ref{TableB}.
	The first matrix in 
	Table \ref{TableB} is the direct sum of a $1 \times 1$ matrix and a Toeplitz matrix
	and is therefore UECSM by \cite[Sect.~2.2]{CCO}.  On the other hand, the second matrix
	in Table \ref{TableB} is not UECSM.  To see this requires a little additional work.
	First note that the lower right $3 \times 3$ block is not UECSM (see Example \ref{ExampleAB} or \cite[Ex.~1]{SNCSO}).
	We next use the fact that a matrix $T$ is UECSM if and only if the external direct sum
	$0 \oplus T$ is UECSM \cite[Lem.~1]{CSPI}.

%

\section{Testing compact operators}\label{SectionVolterra}

	Our final example indicates that the natural infinite-dimensional
	generalization of \texttt{ModulusTest} can sometimes be used to detect
	hidden symmetries in Hilbert space operators.  For instance if $T$ is compact,
	then $T^*T$ and $TT^*$ are diagonalizable selfadjoint operators having the same spectrum
	\cite[Pr.~76]{HalmosHilbert} and hence the proofs of our results go through \emph{mutatis mutandis}.

	\begin{Example}
		We claim that the \emph{Volterra integration operator} $T:L^2[0,1]\to L^2[0,1]$, defined by
		\begin{equation*}
			[Tf](x) = \int_0^x f(y)\, dy,
		\end{equation*}
		is unitarily equivalent to a complex symmetric matrix acting on $l^2(\Z)$.
		Before explicitly demonstrating this with \texttt{ModulusTest}, let us note that neither of the 
		other procedures
		previously available  (\texttt{StrongAngleTest} \cite{Balayan}, \texttt{UECSMTest} \cite{Tener}) 
		are capable of showing this.
		\begin{enumerate}\addtolength{\itemsep}{0.25\baselineskip}
			\item The Volterra operator has no eigenvalues at all (indeed, it is quasinilpotent) and hence no straightforward
				generalization of \texttt{StrongAngleTest} can possibly apply.
				
			\item Since $[T^*f](x) = \int_x^1 f(y)\,dy$, we find that $A = \frac{1}{2}(T+T^*)$
				equals $\frac{1}{2}$ times the orthogonal projection onto the one-dimensional
				subspace of $L^2[0,1]$ spanned by the constant function $1$.  In particular,
				the operator $A$ has the eigenvalue $0$ with infinite multiplicity whence
				no direct generalization of Tener's \texttt{UECSMTest} can possibly apply.
		\end{enumerate}
		On the other hand, the singular values of the Volterra operator are distinct and thus
		\texttt{ModulusTest} applies.
		In fact, the eigenvalues of $T^*T$ and $TT^*$ are 
		\begin{equation*}
			\lambda_n = \frac{2}{(2n+1)\pi}, 
		\end{equation*}	
		for $n = 0,1,2,\ldots$ and corresponding normalized eigenvectors are	
		\begin{equation*}
			u_{n} = \sqrt{2}\cos[(n+\tfrac{1}{2})\pi x], \quad
			v_{n} = \sqrt{2}\sin[(n+\tfrac{1}{2})\pi x].
		\end{equation*}
		These computations are well-known 	\cite[Pr.~188]{HalmosHilbert} and left to the reader
		(a different derivation of these facts can be found in \cite[Ex.~6]{CSO2}).
		An elementary computation now reveals that
		\begin{equation*}
			\inner{u_i,v_j} = 
			\begin{cases}
				\dfrac{ (-1)^{i+j} (2i+1) - (2j+1) }{ \pi( i-j + i^2 -j^2) } & \text{if $i \neq j$}, \\[10pt]
				\dfrac{2}{\pi(1+2i)} & \text{if $i = j$},
			\end{cases}
		\end{equation*}
		from which it is clear that
		\begin{equation}\label{eq-VolterraAlpha}
			\inner{u_i,v_j} = (-1)^{i+j} \inner{u_j,v_i}.
		\end{equation}
		Taking absolute values of the preceding, we see that \eqref{eq-Magnitude} is satisfied.
		Moreover,
		\begin{align*}
			\inner{u_i,v_j} \inner{u_j,v_k} \inner{u_k,v_i}
			&= (-1)^{2(i+j+k)} \inner{u_i,v_k}\inner{u_k,v_j}\inner{u_j,v_i}\\
			&= \inner{u_i,v_k}\inner{u_k,v_j}\inner{u_j,v_i},
		\end{align*}
		whence \eqref{eq-Cocycle} is satisfied.  By Theorem \ref{TheoremMain},
		it follows that the Volterra operator $T$ has a complex symmetric
		matrix representation with respect to some orthonormal basis of $L^2[0,1]$.
		Let us exhibit this explicitly.
		
		Looking at  \eqref{eq-VolterraAlpha} we define $\alpha_n = (-1)^n$ and note that
		\eqref{eq-AlphaCondition} is satisfied for all $i$ and $j$.  We now wish to concretely
		identify the conjugation $C$ on $L^2[0,1]$ which satisfies
		\begin{equation*}
			C( \underbrace{ \cos[ (n+ \tfrac{1}{2})\pi x] }_{u_n} ) = 
			\underbrace{ (-1)^n }_{\alpha_n} \underbrace{ \sin[ (n+ \tfrac{1}{2})\pi x] }_{v_n}
		\end{equation*}
		for $n=0,1,2,\ldots$. Basic trigonometry tells us that
		\begin{align*}
			u_n(1-x)
			&= \cos[ (n+ \tfrac{1}{2})\pi (1-x)]  \\
			&= \cos (n+ \tfrac{1}{2})\pi \cos (n+ \tfrac{1}{2})\pi x + \sin (n+ \tfrac{1}{2})\pi \sin (n+ \tfrac{1}{2})\pi x \\
			&= (-1)^n \sin[ (n+ \tfrac{1}{2})\pi x] 
			= \alpha_n v_n(x) \\
			&= [Cu_n](x)
		\end{align*}
		whence $[Cf](x) = \overline{ f(1-x)}$ for $f \in L^2[0,1]$.  In particular, it is readily verified that $T = CT^*C$
		(see also \cite[Lem.~4.3]{CCO}, \cite[Sect.~4.3]{CSOA}).		
	
		Now observe that $C$ fixes each element of the orthonormal basis
		\begin{equation}\tag{$n \in \Z$}
			e_n = \exp[2 \pi i n (x - \tfrac{1}{2})],
		\end{equation}
		of $L^2[0,1]$ and that
		the matrix for $T$ with respect to this basis is		
		\begin{equation*}
			\left(
			\begin{array}{cccc|c|cccc}
				& \vdots & \vdots & \vdots & \vdots & \vdots & \vdots & \vdots &  \\
				\cdots &   \frac{i}{6 \pi } & 0 & 0 & \frac{i}{6 \pi } & 0 & 0 & 0 &  \cdots \\[3pt]
				\cdots &   0 & \frac{i}{4 \pi } & 0 & -\frac{i}{4 \pi } & 0 & 0 & 0 &  \cdots \\[3pt]
				\cdots &   0 & 0 & \frac{i}{2 \pi } & \frac{i}{2 \pi } & 0 & 0 & 0 &  \cdots \\[3pt]
				 \hline 
				 \cdots  & \frac{i}{6 \pi } & -\frac{i}{4 \pi } & \frac{i}{2 \pi } 
				 & \frac{1}{2} & -\frac{i}{2 \pi } & \frac{i}{4 \pi } & -\frac{i}{6 \pi }& \cdots \\[3pt]
				 \hline
				\cdots &   0 & 0 & 0 & -\frac{i}{2 \pi } & -\frac{i}{2 \pi } & 0 & 0 &  \cdots \\[3pt]
				\cdots &   0 & 0 & 0 & \frac{i}{4 \pi } & 0 & -\frac{i}{4 \pi } & 0 &  \cdots \\[3pt]
				\cdots &   0 & 0 & 0 & -\frac{i}{6 \pi } & 0 & 0 & -\frac{i}{6 \pi } &  \cdots \\[3pt]
				 & \vdots & \vdots & \vdots & \vdots & \vdots & \vdots & \vdots &   \\
			\end{array}
			\right).
		\end{equation*}
		In particular, the Cartesian components $A$ and $B$ of the Volterra operator are
		clearly visible in the preceding matrix representation.
	\end{Example}

\end{document}